\documentclass[twoside,reqno]{amsart}

\usepackage{graphicx}
\usepackage{amsfonts,amssymb,enumitem}
\usepackage{mathtools}
\usepackage{dsfont} 
\usepackage{microtype} 
\usepackage{cleveref}

\numberwithin{equation}{section}

\theoremstyle{plain}
\newtheorem{theorem}{Theorem}[section]
\newtheorem{corollary}[theorem]{Corollary}
\newtheorem{lemma}[theorem]{Lemma}

\newtheorem{Setting}[theorem]{Setting}

\theoremstyle{definition}
\newtheorem{definition}[theorem]{Definition}

\theoremstyle{remark}
\newtheorem{remark}[theorem]{Remark}
\newtheorem{example}[theorem]{Example}

\crefname{theorem}{Theorem}{Theorems}
\crefname{definition}{Definition}{Definitions}
\crefname{remark}{Remark}{Remarks}
\crefname{lemma}{Lemma}{Lemmas}
\Crefname{proposition}{Proposition}{Propositions}
\crefname{corollary}{Corollary}{Corollaries}

\newcommand{\om}{\omega}
\newcommand{\omneu}{\upsilon}
\newcommand{\de}{\theta}
\newcommand{\e}{\varepsilon}
\newcommand{\ee}{\textup{e}}
\newcommand{\Borel}{\mathfrak B(\mathbb R^d)}
\newcommand{\defeq}{\vcentcolon=}
\newcommand{\eqdef}{=\vcentcolon}
\newcommand{\nf}[2]{{#1}/{#2}}
\newcommand{\conste}{\eta_1}
\newcommand{\constz}{\eta_2}
\newcommand{\constd}{\kappa}
\newcommand{\s}{s}
\newcommand{\sset}{K} 
\newcommand{\cset}{F} 
\newcommand{\scontr}{\phi } 
\newcommand{\ccontr}{\psi}
\newcommand{\Scontr}{\Phi}
\newcommand{\Ccontr}{\Psi}
\newcommand{\sratio}{r} 
\newcommand{\edim}{d} 
\newcommand{\hdim}{{\delta}}
\newcommand{\map}{g}
\newcommand{\meas}{\mu}
\newcommand{\ABC}{\mathcal A}
\newcommand{\codes}{\Sigma} 
\newcommand{\met}{\rho} 
\newcommand{\holderexp}{\alpha}
\newcommand{\leb}{\lambda^{\edim}}
\newcommand{\Mink}{\mathcal M}
\newcommand{\arith}{a}
\newcommand{\bepsde}{b_{\e,\de}}
\newcommand{\Lip}{c} 
\newcommand{\low}{\gamma_{_l}} 
\newcommand{\upp}{\gamma_{_{u}}} 
\newcommand{\codeepsde}{\Sigma(\e,\de)}
\newcommand{\diam}{\text{diam}}
\newcommand{\ssc}{\beta}
\newcommand{\gommin}{G_{\om}}

\newcommand{\opens}{O}

\renewcommand{\tilde}{\widetilde}

\DeclareMathOperator*{\wlim}{w-lim}

\crefname{equation}{Equation}{Equations}

\begin{document}

\title[(local) Minkowski Content for a class of self-conformal sets]{Minkowski Content and local Minkowski Content for a class of self-conformal sets}

\author{Uta Freiberg}\address[Uta Freiberg]{Universit\"at Siegen, FB 6 - Mathematik, Walter-Flex-Str. 3, 57068 Siegen, Germany} \email{freiberg@mathematik.uni-siegen.de}   
\author{Sabrina Kombrink}\address[Sabrina Kombrink]{Universit\"at Bremen, Bibliothekstra{\ss}e 1, 28395 Bremen, Germany} \email{kombrink@math.uni-bremen.de}

\begin{abstract}
We investigate (local) Minkowski measurability of $\mathcal C^{1+\alpha}$ images of self-similar sets. We show that (local) Minkowski measurability of a self-similar set $\sset$ implies (local) Minkowski measurability of its image $\cset$ and provide an explicit formula for the (local) Minkowski content of $\cset$ in this case. A counterexample is presented which shows that the converse is not necessarily true. That is, $\cset$ can be Minkowski measurable although $\sset$ is not. However, we obtain that an average version of the (local) Minkowski content of both $\sset$ and $\cset$ always exists and also provide an explicit formula for the relation between the (local) average Minkowski contents of $\sset$ and $\cset$. 
\end{abstract}

\maketitle


\section{Introduction and statement of results}
The Minkowski content is a useful tool for describing the geometric structure of a fractal object. It can be viewed as a beneficial complement to the notion of dimension for the following reason.
It is well known that fractal sets of the same ``fractal'' dimension (such as Minkowski or Hausdorff dimension) can differ significantly in their structure. For example, consider the following two Cantor sets: Subdivide the unit interval $[0,1]$ into seven intervals of same lengths. For the first Cantor set $C_1$ keep the first, third, fifth and seventh interval from the left and repeat the same procedure with the remaining intervals. For the second Cantor set $C_2$ keep at each step the two leftmost and the two rightmost intervals. Then the Minkowski as well as the Hausdorff dimension of $C_1$ and $C_2$ are equal, although the two sets differ significantly in their gap structure. The Minkowski content is capable of detecting this structural difference, as is discussed in \cite{Mandelbrot_Buch,Mandelbrot_lacunarity}, and was proposed therein as a measure of lacunarity for fractal sets. The word lacunarity originates from \texttt{lacuna} which is Latin for gap. According to \cite{Mandelbrot_Buch} ``a fractal is to be called lacunar if its gaps tend to be large, in the sense that they include large intervals (discs, or balls).'' Thus, $C_2$ is more lacunar than $C_1$. This is also reflected by the fact that the average Minkowski content of $C_1$ is greater than that of $C_2$ (see Example \ref{ex:C1C2}).\\

Besides the geometric interpretation, results on the existence of the Minkowski content play an important role with respect to the Weyl-Berry conjecture concerning the asymptotic distribution of the eigenvalues of the Laplacian on domains with fractal boundaries. More precisely, the second term asymptotic is expressed in terms of the Minkowski dimension and the Minkowski content of the boundary of the domain (see Section 4 in \cite{Falconer_Minkowski},\cite{Lapidus_Drum,Levitin} and references given there).

Another motivation for studying the Minkowski content of fractal sets arises from noncommutative geometry. 
In Connes' seminal book \cite{Connes_seminal} the notion of a noncommutative fractal geometry is developed. There, it is shown that the natural analogue of the volume of a compact smooth Riemannian spin manifold for a fractal set in $\mathbb R$ is that of the Minkowski content. This idea is also reflected in the works \cite{FalconerSamuel,Guido_Isola,Samuel}.\\

There are various works available concerning the existence of the Minkowski content. A complete characterisation of Minkowski measurability of fractal strings has been obtained in \cite{LapvF_Springer,LapPom}.
These works, as well as \cite{Falconer_Minkowski}, lead to explicit formulae for the Minkowski content of self-similar subsets of $\mathbb R$ satisfying the open set condition (OSC). Moreover, it is shown that a self-similar subset of $\mathbb R$ which is of zero Lebesgue measure is Minkowski measurable if and only if it is nonlattice in the sense of Definition \ref{def:lattice} (see \cite{LapvF_Springer} and references within). In higher dimensions, Gatzouras \cite{Gatzouras} obtains Minkowski measurability of nonlattice self-similar sets satisfying the OSC and gains explicit formulae for their Minkowski content. 
Assuming certain conditions on the geometric structure of the underlying set, alternative formulae are obtained in \cite{Trken,LapPeaWin} for the nonlattice case. Furthermore, there it is shown that the Minkowski content does not exist in the lattice situation.
For non-Minkowski measurable sets it is worthwhile considering the average Minkowski content, which is defined to be the logarithmic Ces\`aro average (see Definition \ref{def:Minkc}) and has been proven to exist for any self-similar set satisfying the OSC in \cite{Gatzouras}. 

In \cite{Winter_thesis}, the results of \cite{Gatzouras} are generalised in that a localised version of the (average) Minkowski content is examined. This localised version of the (average) Minkowski content, which we call the local (average) Minkowski content (see Definition \ref{def:locMinkc}), is one of the (average) fractal curvature measures which are introduced in \cite{Winter_thesis} and studied for (random) self-similar sets in \cite{Winter_thesis,WinterZaehle,Zaehle_random}. The intention hehind introducing fractal curvature measures was to develop an alternative notion of curvature, since the classical notions do not seem to be appropriate for fractal sets. Moreover, the introduction of fractal curvature measures was motivated by finding geometric characteristics for fractal sets that supplement the notions of dimension.\\

In this paper we are interested in statements on the existence of the (average) Minkowski content --and its local version-- of sets which are more general than self-similar sets, namely self-conformal sets. Self-conformal sets arise as invariant sets of iterated function systems consisting of contracting conformal maps (see for example \cite{MauldinUrbanski}). Some results have already been obtained for these kind of sets. In \cite{KesKom} it is shown that the (local) average Minkowski content of a self-conformal subset of $\mathbb R$ which satisfies the OSC exists and can be calculated explicitly. Moreover, in the nonlattice case, existence and an explicit formula for the (local) Minkowski content have been obtained (we refer to \cite{KesKom} for the explanation what it means for a self-conformal set to be nonlattice and for the explicit formula for the (average) Minkowski content --and its local version).
In this present paper, we extend these examinations to higher dimensions by considering self-conformal sets which arise as images under $\mathcal{C}^{1+\alpha}$-diffeomorphisms of self-similar sets. 
To be more precise, we consider the following setting.
\begin{Setting}\label{setting}
Let $(\mathbb R^{\edim},\met)$ denote the $\edim$-dimensional Euclidean space and for a compact set $\emptyset\neq Y\subset\mathbb R^{\edim}$ and $\e>0$ define $Y_{\e}\defeq\{x\in\mathbb R^{\edim}\mid \met(Y,x)\leq\e\}$ to be the \emph{$\e$-parallel neighbourhood} of $Y$. 
Let $\Scontr\defeq\{\scontr_1,\ldots,\scontr_N\}$ denote an \emph{iterated function system} (IFS) consisting of contracting similarities $\scontr_i\colon \mathbb R^{\edim}\to \mathbb R^{\edim}$, $i\in\{1,\ldots,N\}$, where $N\geq 2$. We require the contraction ratios $\sratio_1,\ldots,\sratio_N$ of $\scontr_1,\ldots,\scontr_N$ to lie in $(0,1)$ and denote by $\sset$ the unique nonempty compact invariant set of $\Scontr$. We assume that $\Scontr$ satisfies the \emph{strong separation condition} (SSC), that is, $\scontr_i \sset\cap\scontr_j \sset=\emptyset$ for each $i\neq j\in\{1,\ldots,N\}$. Note that the SSC implies the OSC, that is, there exists a bounded open nonempty set $\opens\subset\mathbb R^{\edim}$ such that $\scontr_i\opens\subseteq\opens$ for all $i\in\{1,\ldots,N\}$ and $\scontr_i\opens\cap\scontr_j\opens=\emptyset$ for $i\neq j\in\{1,\ldots,N\}$. Associated with such an IFS is the \emph{code space} $\codes\defeq\ABC^{\mathbb N}$, where $\ABC\defeq\{1,\ldots,N\}$ denotes the \emph{alphabet} consisting of $N$ symbols. The code space gives a natural encoding of the invariant set $K$ via the \emph{code map} $\pi\colon\codes\to K$, which maps $\om_1\om_2\cdots\in\codes$ to the unique point in the intersection $\bigcap_{n\in\mathbb N} \phi_{\om_1\cdots\om_n}\sset$, where $\phi_{\om_1\cdots\om_n}\defeq\phi_{\om_1}\circ\cdots\circ\phi_{\om_n}$. 
We also require a conformal diffeomorphism $\map\colon\mathcal{U}\to\mathbb R^{\edim}$ defined on an open domain $\mathcal U$ containing the $1/2$-parallel neighbourhood $\sset_{\nf{1}{2}}$ of $\sset$, where conformal means angle preserving. 
Recall that the Jacobian $\textup{D}\map$ of a conformal map $\map$ at a point $x\in\mathcal U$ can be decomposed into an orthogonal matrix $O(x)$ and a scalar $f(x)$, namely
\[
\textup{D}\map(x)=f(x)\cdot O(x)
\]
(see for example Chapter A.3 in \cite{Lectureshyperbolic}). The \emph{length scaling ratio} of $\map$ at a point $x$ will be denoted by $|\map'(x)|\defeq|f(x)|$. We assume that $|\map'|$ is $\holderexp$-H\"older continuous with $\holderexp>0$ and set $\cset\defeq\map(\sset)$.
Then $\cset$ satisfies
\[
	\cset=\bigcup_{i=1}^N \map\scontr_i\map^{-1} (\cset).
\]
The maps $\ccontr_i\defeq\map\scontr_i\map^{-1}$ for $i\in\ABC$ are not necessarily contractions. However, the $\alpha$-H\"older continuity of $\lvert\map'\rvert$ implies that an iterate $\tilde{\Ccontr}$ of the system $\Ccontr\defeq\{\ccontr_1,\ldots,\ccontr_N\}$ consists solely of contractions. Indeed, $\tilde{\Ccontr}$ is an IFS and $\cset$ is its unique nonempty compact invariant set. Note that the IFS $\tilde{\Ccontr}$ also satisfies the SSC, since $\map$ is a diffeomorphism.
\end{Setting}
 
Crucial for the definition of the (average) Minkowski content --and its local version-- is the notion of the Minkowski dimension.
\begin{definition}[Minkowski dimension]
  For a nonempty compact set $Y\subset\mathbb R^{\edim}$ the \emph{upper} and \emph{lower Minkowski dimensions} are respectively defined to be 
  \begin{align*}
    \overline{\textup{dim}}_M(Y)\defeq\edim-\liminf_{\e\searrow 0}\frac{\ln\leb(Y_{\e})}{\ln\e}\quad\text{and}\quad
    \underline{\textup{dim}}_M(Y)\defeq\edim-\limsup_{\e\searrow 0}\frac{\ln\leb(Y_{\e})}{\ln\e}.
  \end{align*}
	Here, $\leb$ denotes the $d$-dimensional Lebesgue measure. 
  In case the upper and lower Min\-kows\-ki dimensions coincide, we call the common value the \emph{Minkowski dimension} of $Y$ and denote it by $\textup{dim}_M(Y)\eqdef\hdim$.
\end{definition}
\begin{remark}
The (upper and lower) Minkowski dimension coincides with the (upper and lower) box counting dimension (see Proposition 3.2 in \cite{Falconer_Foundation}).
\end{remark}
\begin{definition}[(Average) Minkowski content, Minkowski measurability]
  Let $Y\subset\mathbb R^d$ denote a nonempty compact set whose Minkowski dimension $\hdim$ exists.
\label{def:Minkc}
\begin{enumerate}
\item The \emph{average Minkowski content} of $Y$ is defined by
	\[
  \tilde{\Mink}(Y)\defeq\lim_{T\to 0}\lvert \ln T\rvert^{-1}\int_{T}^1 \e^{\hdim-\edim}\leb(Y_{\e})\frac{\textup{d}\e}{\e},
  \] 
	provided the limit exists. 
\item The \emph{upper} and \emph{lower Minkowski contents} of $Y$ are respectively defined to be
  \[
	\overline{\Mink}(Y)\defeq\limsup_{\e\searrow 0}\e^{\hdim-\edim}\leb(Y_{\e})\quad\text{and}\quad
	\underline{\Mink}(Y)\defeq\liminf_{\e\searrow 0}\e^{\hdim-\edim}\leb(Y_{\e}).
  \]
If $\overline{\Mink}(Y)=\underline{\Mink}(Y)$, then we call the common value the \emph{Minkowski content} of $Y$ and denote it by $\Mink(Y)$. If $\Mink(Y)$ exists, then $Y$ is said to be \emph{Minkowski measurable}.
\end{enumerate}
\end{definition}

Often, not only the global structure of a set is of interest but its local structure is too, since it contains more information on the `texture' of the set itself. This information is reflected by the local (average) Minkowski content, which gives a refinement of the (average) Minkowski content.
\begin{definition}[Local (average) Minkowski content]
Let $Y\subset\mathbb R^d$ denote a nonempty compact set whose Minkowski dimension $\hdim$ exists.
\label{def:locMinkc}
\begin{enumerate}
\item Provided the weak limit of finite Borel measures exists, we define 
\[
	\tilde{\Mink}(Y,\cdot)\defeq\wlim_{T\to 0}\lvert\ln T\rvert^{-1}\int_T^1\e^{\hdim-\edim-1}\leb(Y_{\e}\cap\cdot)\textup{d}\e
\]
to be the \emph{local average Minkowski content} of $Y$.
\item  The \emph{local Minkowski content} $\Mink(Y,\cdot)$ is defined, whenever this weak limit exists, to be the weak limit of finite Borel measures 
\[
	\Mink(Y,\cdot)\defeq\wlim_{\e\to 0}\e^{\hdim-\edim}\leb(Y_{\e}\cap\cdot).
\]
\end{enumerate}
\end{definition}

It is well-known that the Minkowski dimension of self-similar sets satisfying the SSC exists, that it is equal to the Hausdorff dimension and that it is given by the unique solution $s$ of the equation $\sum_{i=1}^N \sratio_i^{s}=1$ (see Theorem 9.3 in \cite{Falconer_Foundation}). Moreover, $\cset$ has the same Minkowski dimension as $\sset$, since $\map$ is a bi-Lipschitz map (see Corollary 2.4 together with Theorem 9.3 in \cite{Falconer_Foundation}).
Thus, the terms from Definitions \ref{def:Minkc} and \ref{def:locMinkc} are defined for such sets. This allows us to formulate our results.
\begin{theorem}[(Average) Minkowski content]
  With the notation of Setting \ref{setting}, let $\hdim$ denote the Minkowski dimension of $\sset$ (and hence $\cset$). Let $\mu_{\hdim}$ denote the normalised $\hdim$-dimensional Hausdorff measure on $\sset$, that is, $\mu_{\hdim}=\mathcal{H}^{\hdim}\vert_{\sset}/\mathcal{H}^{\hdim}(\sset)$, where $\mathcal H^{\hdim}$ denotes the $\hdim$-dimensional Hausdorff measure. Then the following hold.
\label{thm:Minkcresult}
\begin{enumerate}
  \item The average Minkowski contents of $\sset$ and $\cset$ always exist and are positive and finite. Moreover, they satisfy the relation 
  \[
  \tilde{\Mink}(\cset)=\tilde{\Mink}(\sset)\cdot\int_{\sset}\lvert\map'\rvert^{\hdim}\textup{d}\mu_{\hdim}.
  \]
  \item\label{it:Minkcontent} $\cset$ is Minkowski measurable if $\sset$ is Minkowski measurable. In this case we have $\Mink(\sset)=\tilde{\Mink}(\sset)$ and $\Mink(\cset)=\tilde{\Mink}(\cset)$.
\end{enumerate}
\end{theorem}
\begin{theorem}[Local (average) Minkowski content]
With the notation of Setting \ref{setting}, let $\hdim$ denote the Minkowski dimension of $\sset$ (and hence $\cset$) and define $\mu_{\hdim}$ as in Theorem \ref{thm:Minkcresult}. Then the following hold.
\label{thm:locMinkcresult}
\begin{enumerate}
\item\label{locaverage} The local average Minkowski contents of $\sset$ and $\cset$ always exist. Moreover, $\tilde{\Mink}(\cset,\cdot)$ is absolutely continuous with respect to the push forward measure $\map_{\star}\tilde{\Mink}(\sset,\cdot)$ and their Radon-Nikodym derivative is
  \[
  \frac{\textup{d}\tilde{\Mink}(\cset,\cdot)}{\textup{d}\left(\map_{\star}\tilde{\Mink}(\sset,\cdot)\right)}=\lvert\map'\circ\map^{-1}\rvert^{\hdim}.
  \]
\item\label{locexistence} If the local Minkowski content of $\sset$ exists, then the local Minkowski content of $\cset$ exists. Moreover, $\Mink(\sset,\cdot)=\tilde{\Mink}(\sset,\cdot)$ and $\Mink(\cset,\cdot)=\tilde{\Mink}(\cset,\cdot)$.
\end{enumerate} 
\end{theorem}
For subsets of $\mathbb R$ it was shown in \cite{KesKom} that the converse of Theorem \ref{thm:locMinkcresult}(\ref{locexistence}) also holds. Hence the local Minkowski content of $\cset$ exists if and only if the local Minkowski content of $\sset$ exists. However, it is important to remark that the converse of Theorem \ref{thm:Minkcresult}(\ref{it:Minkcontent}) is not true in general. 
To illustrate this, we present the following example, which is originally given as Example 2.15(iii) in \cite{KesKom}.
\begin{example}[Kesseb\"ohmer/Kombrink]
  Let $\sset\subset\mathbb R$ denote the middle third Cantor set and let $\mu$ denote the normalised $(\ln2/\ln3)$-dimensional Hausdorff measure on $\sset$. Let $f\colon\mathbb R\to\mathbb R$ denote the Devil's staircase function given by $f(r)\defeq\mu((-\infty,r])$, define the function $\map\colon\mathbb R\to\mathbb R$ by $\map(x)\defeq\int_{-\infty}^{x}(f(y)+1)^{-\ln3/\ln2}\textup{d}y$ and set $\cset\defeq\map(\sset)$. Then $\cset$ is Minkowski measurable although $\sset$ is not.
\end{example}

Next, we present some results from \cite{Gatzouras,Winter_thesis}, which in tandem with Theorems \ref{thm:Minkcresult} and \ref{thm:locMinkcresult} allow us to deduce explicit formulae for the (average) Minkowski content --and its local version-- for $\mathcal{C}^{1+\alpha}$-diffeomorphic images of self-similar sets.
For the statement of these theorems, we require the following definition.
\begin{definition}[(Non)lattice, scaling function]
  Fix the notation of Setting \ref{setting}. The iterated function system $\Scontr$ is said to be \emph{lattice} if there exists an $\arith>0$ such that $\ln\sratio_i\in\arith\mathbb Z$ for all $i\in\ABC$. If $\arith>0$ is maximal with this property, then $\Scontr$ is called \emph{$\arith$-lattice}. If, on the other hand, no such $\arith>0$ exists, then $\Scontr$ is called \emph{nonlattice}. We use the terms lattice and nonlattice also for the invariant set $\sset$, if the associated IFS $\Scontr$ is lattice or nonlattice respectively.
	Furthermore, the \emph{scaling function} $R_{\edim}(K,\cdot)\colon(0,\infty)\to\mathbb R$ of $\sset$ is defined by setting
\label{def:lattice}
\begin{equation*}
  R_{\edim}(\sset,\e)\defeq\leb(\sset_{\e})-\sum_{i=1}^{N}\mathds 1_{(0,\sratio_i]}(\e)\leb((\scontr_i\sset)_{\e}).
\end{equation*}
\end{definition}
We remark that for $\e>0$ small enough, $R_{\edim}(\sset,\e)$ is equal to $-\leb(\bigcup_{i\neq j\in\ABC}(\scontr_i\sset)_{\e}\cap(\scontr_j\sset)_{\e})$ and thus describes the volume of the overlap of sets of the form $(\scontr_i\sset)_{\e}$ for $i\in\ABC$. (This follows from an inclusion-exclusion argument.)

\begin{theorem}[Gatzouras]
Assume that the conditions of Setting \ref{setting} are satisfied. Let $\hdim$ denote the Minkowski dimension of the self-similar set $\sset$. Then the following hold.
\label{thm:Gatzouras}
\begin{enumerate}
  \item\label{ss:averageMink} The average Minkowski content $\tilde{\Mink}(\sset)$ of $\sset$ exists, is positive and given by
    \begin{equation*}
      \tilde{\Mink}(\sset)=-\left(\sum_{i=1}^N\sratio_i^{\hdim}\ln\sratio_i\right)^{-1}\int_0^1\e^{\hdim-\edim-1}R_{\edim}(\sset,\e)\textup{d}\e.
    \end{equation*}
  \item\label{ss:nonlatticeMink} If $\Scontr$ is nonlattice, then the Minkowski content of $\sset$ exists and coincides with the average Minkowski content, that is $\Mink(\sset)=\tilde{\Mink}(\sset)$.
\end{enumerate}
\end{theorem}
The above theorem was originally given in Theorem 2.3 in \cite{Gatzouras} but is presented in the form of Theorem 2.3.10 of \cite{Winter_thesis}.

\begin{remark}
 The factor $-\sum_{i=1}^N\sratio_i^{\hdim}\ln\sratio_i$ multiplied with $\hdim$ coincides with the measure theoretical entropy of the shift-map with respect to the unique shift-invariant Gibbs measure $\mu_{-\hdim\xi}$ for the potential function $-\hdim\xi$. Here the \emph{geometric potential function} $\xi\colon\codes\to\mathbb R$ is defined by $\xi(\om)\defeq-\ln\sratio_{\om_1}$ for $\om=\om_1\om_2\cdots\in\codes$. The quantity $-\hdim\sum_{i=1}^N\sratio_i^{\hdim}\ln\sratio_i$ is also known as the entropy of the probability distribution $(\sratio_1^{\hdim},\ldots,\sratio_N^{\hdim})$. For further explanation of these terms see \cite{Bowen_equilibrium}.
\end{remark}
\begin{theorem}[Winter]
Assume that the conditions of Setting \ref{setting} hold. Denote by $\hdim$ the Minkowski dimension of the self-similar set $\sset$ and let $\mu_{\hdim}$ denote the normalised $\hdim$-dimensional Hausdorff measure on $\sset$. Then the following hold.
\label{thm:Winter}
\begin{enumerate}
  \item\label{ss:localaverage} The local average Minkowski content of $\sset$ exists and is given by
    \begin{equation*}
    \tilde{\Mink}(\sset,\cdot)=\tilde{\Mink}(\sset)\mu_{\hdim}(\cdot).
    \end{equation*}
  \item\label{ss:localnonlattice} If $\Scontr$ is nonlattice, then the local Minkowski content exists, and we have $\Mink(\sset,\cdot)=\tilde{\Mink}(\sset,\cdot)$.
\end{enumerate}
\end{theorem}
Theorem \ref{thm:Winter} corresponds to Theorem 2.5.1 in \cite{Winter_thesis}.\\

Note that all the results from Theorems \ref{thm:Gatzouras} and \ref{thm:Winter} actually hold under the weaker OSC.
Moreover, under certain additional assumptions an alternative formula for the (average) Minkowski content of $\sset$ can be found in \cite{Trken,LapPeaWin}.\\

Let us return to the two Cantor sets $C_1$ and $C_2$ which were described at the beginning of the introduction. An application of the theorem of Gatzouras (Theorem \ref{thm:Gatzouras}) yields explicit values for their average Minkowski contents in the following way.
\begin{example}
	Recall the construction of the two Cantor sets $C_1$ and $C_2$ from the beginning of the introduction. 
	$C_1$ is the invariant set of the iterated function system $\Scontr\defeq\{\scontr_1,\ldots,\scontr_4\}$, where $\scontr_i(x)=x/7+2(i-1)/7$ for $i\in\{1,\ldots,4\}$. It can be easily verified that the IFS $\Scontr$ satisfies the conditions of Setting \ref{setting} and that the Minkowski dimension of $C_1$ is equal to $\hdim=\ln 4/\ln 7$. An application of Theorem \ref{thm:Gatzouras} now yields that
\label{ex:C1C2}
	\[
		\tilde{\Mink}(C_1)=\frac{3}{2}\cdot\frac{2^{-\hdim}}{(1-\hdim)\hdim\ln 7}.
	\]
	Using the fact that Theorem \ref{thm:Gatzouras} also holds under the weaker OSC, we likewise obtain that 
	\[
		\tilde{\Mink}(C_2)=\frac{3^{\hdim}}{2}\cdot\frac{2^{-\hdim}}{(1-\hdim)\hdim\ln 7},
	\]
	where $\hdim=\ln 4/\ln 7$ is the Minkowski dimension of $C_2$. Thus, $\tilde{\Mink}(C_1)>\tilde{\Mink}(C_2)$.
\end{example}

Combining our Theorem \ref{thm:Minkcresult} with the results from Theorem \ref{thm:Gatzouras} we immediately obtain the following explicit formulae for the (average) Minkowski content of the $\mathcal C^{1+\alpha}$ image $\cset$.
\begin{corollary}
	With the notation of Setting \ref{setting}, let $\hdim$ denote the Minkowski dimension of $\sset$ (and hence $\cset$). Further, denote the scaling function of $\sset$ by $R_d(\sset,\cdot)$ and let $\mu_{\hdim}$ denote the $\hdim$-dimensional normalised Hausdorff measure on $\sset$. Then the following hold.
\label{formulaeMinkc}
  \begin{enumerate}
  \item The average Minkowski content $\tilde{\Mink}(F)$ of $\cset$ exists, is positive and is given by
    \begin{equation}\label{eq:formulaaverage}
      \tilde{\Mink}(\cset)=-\left(\sum_{i=1}^N\sratio_i^{\hdim}\ln\sratio_i\right)^{-1}\int_0^1\e^{\hdim-\edim-1}R_d(\sset,\e)\textup{d}\e\cdot\int_{\sset}\lvert\map'\rvert^{\hdim}\textup{d}\mu_{\hdim}.
    \end{equation}
  \item If $\Scontr$ is nonlattice, then the Minkowski content of $\cset$ exists and coincides with the average Minkowski content, that is $\Mink(\cset)=\tilde{\Mink}(\cset)$.
	\end{enumerate}
\end{corollary}

Combining Theorem \ref{thm:locMinkcresult} with Theorem \ref{thm:Winter} we obtain the following corollary for the local (average) Minkowski content.
\begin{corollary}
	With the notation of Setting \ref{setting}, denote by $\hdim$ the Minkowski dimension of $\sset$ (and hence $\cset$) and let $\mu_{\hdim}$ denote the $\hdim$-dimensional normalised Hausdorff measure on $\sset$. Then the following hold.
\label{cor:local}
\begin{enumerate}
  \item The local average Minkowski content $\tilde{\Mink}(\cset,\cdot)$ of $\cset$ exists and satisfies
    \[
			\frac{\textup{d}\tilde{\Mink}(\cset,\cdot)}{\textup{d}\left(\map_{\star}\mu_{\hdim}\right)(\cdot)}
			=\frac{\lvert\map'\circ\map^{-1}\rvert^{\hdim}}{\int_{\sset}\lvert\map'\rvert^{\hdim}\textup{d}\mu_{\hdim}}\cdot \tilde{\Mink}(\cset).
    \]
   \item\label{formulae_nonlattice_average} If $\Scontr$ is nonlattice, then the local Minkowski content $\Mink(F,\cdot)$ of $\cset$ exists and is equal to $\tilde{\Mink}(\cset,\cdot)$.
  \end{enumerate}
\end{corollary}
Observe that the measure $\mu$ given by $\frac{\textup{d}\mu}{\textup{d}\left(\map_{\star}\mu_{\hdim}\right)}=\frac{\lvert\map'\circ\map^{-1}\rvert^{\hdim}}{\int_{\sset}\lvert\map'\rvert^{\hdim}\textup{d}\mu_{\hdim}}$ coincides with the $\hdim$-conformal measure of the IFS $\tilde{\Ccontr}\eqdef\{\tilde{\ccontr}_1,\ldots,\tilde{\ccontr}_M\}$, where $\tilde{\Ccontr}$ is defined as in Setting \ref{setting} and $M\in\mathbb N$. Here the \emph{$\hdim$-conformal measure} of $\tilde{\Ccontr}$ is the unique probability measure $\mu$ supported on $\cset$, which satisfies
\[
	\mu(\tilde{\ccontr}_i B)=\int_{B}\lvert\tilde{\ccontr}'_i\rvert^{\hdim}\textup{d}\mu
\]
for all $i\in\{1,\ldots,M\}$ and all Borel sets $B\subset\mathbb R^{\edim}$. For more details about this measure, see, for example \cite{MauldinUrbanski}.

\begin{remark}
The results of this paper are concerned with self-conformal sets which arise as $\mathcal C^{1+\alpha}$-images of self-similar sets. For general self-conformal sets in $\mathbb R^{\edim}$, which cannot necessarily be obtained in this way, our results and the ones presented in \cite{KesKom} for subsets of $\mathbb R$ suggest that the (average) local Minkowski content is a constant multiple of the $\hdim$-conformal measure, whenever it exists. This has recently been obtained under certain geometric assumptions and will be presented in a forthcoming paper by the second author, where the dichotomy of lattice versus nonlattice will also be discussed.
\end{remark}

\section{Proofs}
Observe that Theorem \ref{thm:Minkcresult} follows immediately from Theorem \ref{thm:locMinkcresult}. Thus, in this section we exclusively deal with the proof of Theorem \ref{thm:locMinkcresult}. The proofs of the first and second part of Theorem \ref{thm:locMinkcresult} differ quite significantly. However, certain tools are used in both proofs and these tools are presented in Lemmas \ref{lem:bdp} to \ref{lem:disj}. Before turning to them, let us fix some notation.\\

As described in Setting \ref{setting}, let $\sset$ denote the self-similar set which is generated by the iterated function system $\Scontr\defeq\{\scontr_1,\ldots,\scontr_N\}$ consisting of contracting similarities with contraction ratios $\sratio_1,\ldots,\sratio_N$. We assume without loss of generality that $\diam(\sset)=1$.

\paragraph{The Code Space $\Sigma$.}
Recall that we refer to the set $\ABC\defeq\{1,\ldots,N\}$ as the alphabet and let $\ABC^n$ denote the space of words of length $n\in\mathbb N$ over $\ABC$. Furthermore, let $\ABC^*\defeq\bigcup_{n\in\mathbb N\cup\{0\}} \ABC^n$ denote the space of all finite words over $\ABC$ including the empty word $\emptyset$ and recall that $\codes\defeq\ABC^{\mathbb N}$ denotes the code space which represents the set of infinite words over $\ABC$.
For a finite word $\om\in\ABC^*$ its length is denoted by $n(\om)$. For $\om\defeq\om_1\cdots\om_n\in\ABC^*$ we set $\phi_{\om}\defeq\phi_{\om_1}\circ\cdots\circ\phi_{\om_n}$, $\sratio_\om\defeq\sratio_{\om_1}\cdots\sratio_{\om_n}$ and define
$[\om]\defeq\{\overline{\om}\in\codes\mid\overline{\om}_i=\om_i\ \text{for all}\ i\in\{1,\ldots,n(\om)\}\}$ to be the \emph{$\om$-cylinder set}.
Moreover for $\om\defeq\om_1\om_2\cdots\in\codes$ and $n\in\mathbb N$ we denote the initial word of length $n$ of $\om$ by $\om\vert_n\defeq\om_1\om_2\cdots\om_n$. 
Finally, we set $\sratio_{\min}\defeq\min\{\sratio_1,\ldots,\sratio_N\}$.

\paragraph{The Word Space $\codeepsde$.}
We denote the minimal length scaling ratio of $\map$ on the $(1/2)$-parallel neighbourhood $\sset_{1/2}$ of $\sset$ by
\begin{equation*}
  \low\defeq\min_{x\in\sset_{\nf{1}{2}}}|g'(x)|.
\end{equation*}
Since $\map$ is a diffeomorphism with domain $\mathcal U\supset\sset_{1/2}$, we have that $\low>0$. Recall that $\map\in\mathcal C^{1+\alpha}(\mathcal U)$ and denote by $\Lip$ the H\"older constant of $\lvert\map'\rvert$.
For $\de>0$ and $\e\geq 0$ set 
\begin{align*}
\bepsde&\defeq \left(\frac{\de \low}{\Lip}\right)^{\nf{1}{\holderexp}}-2\frac{\e}{\low}\qquad\text{and}\\
\codeepsde&\defeq \left\{\om\in\ABC^*\mid \sratio_{\om}\leq\bepsde\ \text{and}\ \sratio_{\om\vert_{n(\om)-1}}>\bepsde\right\}.
\end{align*}
The family $\codeepsde$ (and in particular $\bepsde$) is constructed in such a way that 
\begin{enumerate}
	\item a powerful bounded distortion lemma holds for $\lvert\map'\rvert$ on $(\e/\low)$-neighbourhoods of $\scontr_{\om}\sset$ for $\om\in\codeepsde$ and sufficiently small $\e\geq 0$ (see Lemma \ref{lem:bdp}) and
	\item $\sset_{\e/\low}$ can be written as a disjoint union of the sets $(\scontr_{\om}\sset)_{\e/\low}$, where the union ranges over $\om\in\codeepsde$ (see Lemmas \ref{lem:union} and \ref{lem:disj}).
\end{enumerate}

\begin{lemma}[Bounded Distortion Lemma]
  For $\de>0,\ 0\leq\e\leq \frac{\low}{2}$ and an arbitrary $\om\in\codeepsde$ we have that
\label{lem:bdp}
  \[
  (1+\de)^{-1}\leq\frac{|\map'(x)|}{|\map'(y)|}\leq 1+\de\qquad\text{for all}\ x,y\in (\scontr_{\om}\sset)_{\nf{\e}{\low}}.
  \]
\end{lemma}
\begin{proof}
  Since $\om$ lies in $\codeepsde$, the diameter of the set $(\scontr_{\om}\sset)_{\nf{\e}{\low}}$ satisfies
  \[
  \diam(\scontr_{\om}\sset)_{\nf{\e}{\low}}
  =\sratio_{\om}\underbrace{\diam(\sset)}_{=1}+2\frac{\e}{\low}
  \leq \bepsde+2\frac{\e}{\low}
  =\left(\frac{\de \low}{\Lip}\right)^{\nf{1}{\holderexp}}.
  \]
  Recalling that $\lvert\map'\rvert$ is $\alpha$-H\"older continuous with H\"older constant $\Lip$, we hence have 
  \[
  \big\lvert |\map'(x)|-|\map'(y)| \big\rvert\leq\Lip|x-y|^{\holderexp}\leq\de\low\qquad\text{for all}\ x,y\in (\scontr_{\om}\sset)_{\nf{\e}{\low}}.
  \]
  Thus, 
  \[
  \frac{|\map'(x)|}{|\map'(y)|}\leq\frac{\big\lvert|\map'(x)|-|\map'(y)|\big\rvert}{|\map'(y)|}+1\leq\de+1
  \]
	for all $x,y\in (\scontr_{\om}\sset)_{\nf{\e}{\low}}$.
  The second inequality follows on interchanging the roles of $x$ and $y$.
\end{proof}
\begin{lemma}
  For $\de>0$ and $0\leq\e<\frac{\low}{2}\left(\frac{\de \low}{\Lip}\right)^{\nf{1}{\holderexp}}$ we have that
	\label{lem:union}	
  \[
  \sset=\bigcup_{\om\in\codeepsde}\scontr_{\om}\sset.
  \]
\end{lemma}
\begin{proof}
  The condition $0\leq\e<\frac{\low}{2}\left(\frac{\de \low}{\Lip}\right)^{\nf{1}{\holderexp}}$ implies that $\bepsde=\left(\frac{\de \low}{\Lip}\right)^{\nf{1}{\holderexp}}-2\frac{\e}{\low}>0$.
  Therefore for every $\om\in\codes$ there exists an $n\in\mathbb N$ such that $\sratio_{\om|_n}\leq\bepsde$.
\end{proof}

For the following lemma, we define 
\[
	\met(Y,Z)\defeq\min_{y\in Y}\met(y,Z)\defeq\min_{y\in Y}\min_{z\in Z}\met(y,z)
\]
for compact sets $Y,Z\subset\mathbb R^{\edim}$. Moreover, we set 
\[
	\ssc\defeq \min_{i\neq j\in\ABC}\met(\scontr_i\sset,\scontr_j\sset)/2
\]
and remark that $\ssc$ is positive because $\Scontr$ satisfies the SSC.
\begin{lemma}\label{lem:disj}
  For $\de>0$ and $0\leq\e<\e_0(\de)\defeq\low\cdot\sratio_{\min}\cdot\left(\frac{\de\low}{\Lip}\right)^{\nf{1}{\holderexp}}\cdot\frac{\ssc}{1+2\ssc\sratio_{\min}}$ the elements of $\{(\scontr_{\om}\sset)_{\nf{\e}{\low}}\mid \om\in\codeepsde\}$ are pairwise disjoint for distinct $\om\in\codeepsde$, that is
  \[
  (\scontr_{\om}\sset)_{\nf{\e}{\low}}\cap (\scontr_{\omneu}\sset)_{\nf{\e}{\low}}=\emptyset\qquad\text{for all}\ \om\neq\omneu\in\codeepsde.
  \]
\end{lemma}
\begin{proof}
  Note that the cardinality of $\codeepsde$ is finite. Therefore, there exists $\om'\in\codeepsde$ satisfying $\sratio_{\om'}\leq\sratio_{\om}$ for all $\om\in\codeepsde$.
  Hence, for $\om\neq\omneu\in\codeepsde$ we have that
  \begin{align*}
    \met(\scontr_{\om}\sset,\scontr_{\omneu}\sset)
    &\geq \sratio_{\om'}\cdot 2\ssc
    \geq \sratio_{\min}\bepsde\cdot 2\ssc
    =\sratio_{\min}\left(\left(\frac{\de \low}{\Lip}\right)^{\nf{1}{\holderexp}}-2\frac{\e}{\low}\right)\cdot 2\ssc\\
    &>\sratio_{\min}\left(\frac{(1+2\ssc\sratio_{\min})}{\low\sratio_{\min}\ssc}\e-2\frac{\e}{\low}\right)\cdot 2\ssc
    =2\frac{\e}{\low},
  \end{align*}
	which implies the assertion.
\end{proof}

\begin{proof}[Proof of Theorem \ref{thm:locMinkcresult}(\ref{locexistence})]
	That the existence of $\Mink(\sset,\cdot)$ (resp. $\Mink(\cset,\cdot)$) implies the existence of $\tilde{\Mink}(\sset,\cdot)$ (resp. $\tilde{\Mink}(\cset,\cdot)$) can be easily seen, since $\tilde{\Mink}(\sset,\cdot)$ (resp. $\tilde{\Mink}(\cset,\cdot)$) is the Ces\`aro-average of $\Mink(\sset,\cdot)$ (resp. $\Mink(\cset,\cdot)$).
	Thus, it only remains to show that 
	\begin{equation}\label{eq:pushfwd}
		\wlim_{\e\to 0}\e^{\hdim-\edim}\leb(\cset_{\e}\cap\cdot)=\meas(\cdot),
	\end{equation}
	where $\meas$ denotes the measure given by 
	\begin{equation*}
		\frac{\textup{d}\meas}{\textup{d}\left(\map_{\star}\Mink(\sset,\cdot)\right)}=\lvert\map'\circ\map^{-1}\rvert^{\hdim}.
	\end{equation*}
	By the Portmanteau Theorem, this is equivalent to the following. For every sequence $(\e_n)_{n\in\mathbb N}$ of positive real numbers converging to 0 we have that
  \begin{enumerate}[label={(\alph*)}]
  \item $\displaystyle{\lim_{n\to\infty}\e_n^{\hdim-\edim}\leb(\cset_{\e_n}\cap\mathbb R^{\edim})=\meas(\mathbb R^{\edim})}$ and\label{condi}
  \item for every closed set $A\subseteq\mathbb R^{\edim}$,
    \[\limsup_{n\to\infty}\e_n^{\hdim-\edim}\leb(\cset_{\e_n}\cap A)\leq\meas(A).\]\label{condii}
  \end{enumerate}
  We start by showing Condition \ref{condii}. Let $(\e_n)_{n\in\mathbb N}$ be an arbitrary sequence of positive real numbers converging to 0 and fix a closed set $A\subseteq\mathbb R^{\edim}$. Fix $\de>0$ and set 
	\begin{equation}\label{eq:tildee}
		\tilde{\e}\defeq\min\left\{\e_0(\de),\frac{\low}{2}\left(\frac{\de\low}{\Lip}\right)^{1/\alpha}\right\},
	\end{equation}
	where $\e_0(\de)$ is defined as in Lemma \ref{lem:disj}. Choose $n_0(\de)\in\mathbb N$ sufficiently large that for all $n\geq n_0(\de)$ we have $\e_n<\tilde{\e}$. From here on, assume that $n\geq n_0(\de)$. For $\om\in\ABC^*$ and $\e\geq 0$ define
	\[
		\gommin(\e)\defeq\min\left\{\lvert\map'(x)\rvert\mid x\in(\scontr_{\om}\sset)_{\e/\low}\right\}.
	\]
	Applying Lemmas \ref{lem:bdp} and \ref{lem:union} yields the following.
  \begin{align}
    \leb(\cset_{\e_n}\cap A)
    &\leq \sum_{\om\in\codes(\e_n,\de)}\leb\left((\map\scontr_{\om}\sset)_{\e_n}\cap A\right)\nonumber\\
		&\leq \sum_{\om\in\codes(\e_n,\de)}\leb\left(\map((\scontr_{\om}\sset)_{\e_n/\gommin(\e_n)}\cap\map^{-1} A)\right)\nonumber\\
    &\leq\sum_{\om\in\codes(\e_n,\de)}\leb\left((\scontr_{\om}\sset)_{\e_n/\gommin(\e_n)}\cap\map^{-1} A\right)\cdot\gommin(\e_n)^{\edim}(1+\de)^{\edim}.\label{eq:loc1}
  \end{align}
	To bound this latter quantity, let us focus on the term $\leb((\scontr_{\om}\sset)_{\e_n/\gommin(\e_n)}\cap\map^{-1}A)$ for $\om\in\codes(\e_n,\de)$.
	Set 
	\[
		D\defeq\low\cdot\min\{\met(\scontr_{\om}\sset,\scontr_{\omneu}\sset)/2\mid\om\neq\omneu\in\codes(\tilde{\e},\de)\}
	\]
	and observe that Lemma \ref{lem:disj} implies that $\e_n/\gommin(\e_n)<D/\low$. Thus, for all $\om\in\codes(\e_n,\de)$, we have that
	\begin{equation}\label{eq:D}
		(\scontr_{\om}\sset)_{\e_n/\gommin(\e_n)}=\sset_{\e_n/\gommin(\e_n)}\cap(\scontr_{\om}\sset)_{D/\low}.
	\end{equation}
	Moreover, $B\defeq\map^{-1}A\cap(\scontr_{\om}\sset)_{D/\low}$ is closed since $\map$ is a diffeomorphism.
  By the hypotheses, $\Mink(\sset,\cdot)\defeq\wlim_{n\to\infty}\e_n^{\hdim-\edim}\leb(\sset_{\e_n}\cap\cdot)$ exists, and so the Portmanteau Theorem and \cref{eq:D} imply that
  \begin{align*}
   &\limsup_{n\to\infty}\left(\frac{\e_n}{\gommin(\e_n)}\right)^{\hdim-\edim}\leb((\scontr_{\om}\sset)_{\e_n/\gommin(\e_n)}\cap\map^{-1}A)\\
	=\ &\limsup_{n\to\infty}\left(\frac{\e_n}{\gommin(\e_n)}\right)^{\hdim-\edim}\leb(\sset_{\e_n/\gommin(\e_n)}\cap B)\\
	\leq\ &\Mink(\sset,B).
  \end{align*}
  Hence, for all $\kappa>0$ there exists an $\overline{n}\in\mathbb N$ such that for all $n\geq \overline{n}$ we have that
  \begin{equation}\label{eq:supkappa}
  \sup_{k\geq n}\left(\frac{\e_k}{\gommin(\e_k)}\right)^{\hdim-\edim}\leb((\scontr_{\om}\sset)_{\e_k/\gommin(\e_k)}\cap \map^{-1}A)\leq\Mink(\sset,B)+\kappa.
  \end{equation}
   From \cref{eq:loc1,eq:supkappa} we now obtain that
  \begin{align*}
    &\e_n^{\hdim-\edim}\leb(\cset_{\e_n}\cap A)\\
    \leq\ & \sum_{\om\in\codes(\e_n,\de)}\left(\frac{\e_n}{\gommin(\e_n)}\right)^{\hdim-\edim}\leb((\scontr_{\om}\sset)_{\e_n/\gommin(\e_n)}\cap\map^{-1}A)\cdot\gommin(\e_n)^{\hdim}(1+\de)^{\edim}\\
    \leq\ & \sum_{\om\in\codes(\e_n,\de)}\left(\Mink(\sset,\map^{-1}A\cap (\scontr_{\om}\sset)_{D/\low})+\kappa\right)\cdot\gommin(\e_n)^{\hdim}(1+\de)^{\edim}.
  \end{align*}
  By Theorems \ref{thm:Gatzouras} and \ref{thm:Winter} we know that $\Mink(\sset,\cdot)=\Mink(\sset)\mu_{\hdim}(\cdot)$.
Thus, the support of $\Mink(\sset,\cdot)$ is $\sset$ and the definition of $D$ implies that $\Mink(\sset,\map^{-1}A\cap (\scontr_{\om}\sset)_{D/\low})=\Mink(\sset,\map^{-1}A\cap \scontr_{\om}\sset)$. Therefore, using the bounded distortion lemma (Lemma \ref{lem:bdp}) we conclude the following.
\begin{align*}
  &\e_n^{\hdim-\edim}\leb(\cset_{\e_n}\cap A)\\
  \leq\ & \int_{\map^{-1}A}\lvert\map'\rvert^{\hdim}\textup{d}\Mink(\sset,\cdot)(1+\de)^{\edim}+ \sum_{\om\in\codes(\e_n,\de)}\gommin(\e_n)^{\hdim}\kappa(1+\de)^{\edim}\\
  \leq\ &\int_{A}\lvert\map'\circ\map^{-1}\rvert^{\hdim}\textup{d}\left(\map_{\star}\Mink(\sset,\cdot)\right)(1+\de)^{\edim}+ \sum_{\om\in\codes(0,\de)}\gommin(0)^{\hdim}\kappa(1+\de)^{\edim+\hdim}.
\end{align*}
Finally, since the expression in the last line does not depend on $n$, we can take the limits as $\kappa\to 0$ and $\de\to 0$ to obtain
\begin{equation*}
  \limsup_{n\to\infty}\e_n^{\hdim-\edim}\leb(\cset_{\e_n}\cap A)
  \leq \int_A \lvert\map'\circ\map^{-1}\rvert^{\hdim}\textup{d}\left(\map_{\star}\Mink(\sset,\cdot)\right)=\meas(A).
\end{equation*}
This shows that Condition \ref{condii} is satisfied.\\ 

Now that Condition \ref{condii} is verified, to obtain Condition \ref{condi} it suffices to show that for every sequence $(\e_n)_{n\in\mathbb N}$ of positive real numbers converging to 0 we have $\liminf_{n\to\infty}\e_n^{\hdim-\edim}\leb(\cset_{\e_n})\geq\meas(\mathbb R^{\edim})$. 
To that end, fix such a sequence $(\e_n)_{n\in\mathbb N}$ and an arbitrary $\de>0$. By our hypotheses, $\lim_{\e\to 0}\e^{\hdim-\edim}\leb(\sset_{\e})=\Mink(\sset)$ and so for all $\kappa>0$ there exists an $\overline{n}\in\mathbb N$ such that for all $n\geq \overline{n}$ and all $\om\in\codes(\e_n,\de)$ we have that
\begin{equation}\label{eq:MinkBruch}
	\left\lvert\left(\frac{\e_n}{\sratio_{\om}\gommin(\e_n)(1+\de)}\right)^{\hdim-\edim}\leb(\sset_{\e_n/(\sratio_{\om}\gommin(\e_n)(1+\de))})-\Mink(\sset)\right\rvert<\kappa.
\end{equation}
Recall the definition of $\tilde{\e}$ from \cref{eq:tildee} and choose $n_0(\de)\geq \overline{n}$ sufficiently large that $\e_n<\tilde{\e}$ for all $n\geq n_0(\de)$. Assume that $n\geq n_0(\de)$ from here on. Lemmas \ref{lem:union} and \ref{lem:disj} together with the inequality given in \cref{eq:MinkBruch} imply that
\begin{align*}
  \e_n^{\hdim-\edim}\leb(\cset_{\e_n})
  &= \e_n^{\hdim-\edim}\sum_{\om\in\codes(\e_n,\de)}\leb((\map\scontr_{\om}\sset)_{\e_n})\\
  &\geq \e_n^{\hdim-\edim}\sum_{\om\in\codes(\e_n,\de)}\gommin(\e_n)^{\edim}\sratio_{\om}^{\edim}\leb(\sset_{\nf{\e_n}{(\sratio_{\om}\gommin(\e_n)(1+\de))}})\\
  &\geq \sum_{\om\in\codes(\e_n,\de)}\gommin(\e_n)^{\hdim}\sratio_{\om}^{\hdim} (1+\de)^{\hdim-\edim}\left(\Mink(\sset)-\kappa\right)\\
  &\geq \sum_{\om\in\codes(0,\de)}\gommin(0)^{\hdim}\sratio_{\om}^{\hdim} (1+\de)^{-\edim}\left(\Mink(\sset)-\kappa\right),
\end{align*}
where the last inequality is a consequence of Lemma \ref{lem:bdp}. Since $\de$ was arbitrarily chosen, the above inequality holds for all $\de>0$.
Having $\lim_{\de\to 0} \sum_{\om\in\codes(0,\de)}\gommin(0)^{\hdim}\sratio_{\om}^{\hdim}=\int_{\sset}|\map'|^{\hdim}\,\text{d}\mu_{\hdim}$, we conclude by taking the limit as $\de$ tends to $0$ that
\begin{equation*}
  \liminf_{n\to\infty} \e_n^{\hdim-\edim}\leb(\cset_{\e_n})
  \geq \int_{\sset}|\map'|^{\hdim}\,\text{d}\mu_{\hdim}\,\cdot\,\left(\Mink(\sset)-\kappa\right) \qquad \text{for all}\ \kappa>0.
\end{equation*} 
Hence,
\begin{equation*}
  \liminf_{n\to\infty}\e_n^{\hdim-\edim}\leb(\cset_{\e_n})
  \geq \int_{\sset}|\map'|^{\hdim}\,\text{d}\mu_{\hdim}\,\cdot\,\Mink(\sset).
\end{equation*}
\end{proof}

Our next aim is to prove Theorem \ref{thm:locMinkcresult}(\ref{locaverage}). The first step in this direction is the following definition. 
An \emph{intersection stable generator} of $\Borel$ is defined to be a collection of sets $\mathcal E\subset\Borel$ such that the smallest $\sigma$-algebra containing $\mathcal E$ coincides with $\Borel$ and such that the intersection of any two elements of $\mathcal E$ again is an element of $\mathcal E$. 
In the proof of Theorem \ref{thm:locMinkcresult}(\ref{locaverage}) we are going to use the fact that two Borel measures which coincide on an intersection stable generator of the Borel $\sigma$-algebra $\Borel$ coincide on $\Borel$. 
The intersection stable generator we use is constructed as follows. First, recall that the SSC implies the OSC and that the OSC was proven to be equivalent to the strong open set condition (SOSC) for self-similar subsets of $\mathbb R^{\edim}$ in \cite{Schief}. An iterated function system $\Scontr\defeq\{\scontr_1,\ldots,\scontr_N\}$ with invariant set $\sset$ satisfies the SOSC if there exists a nonempty bounded open set $\opens\subseteq\mathbb R^{\edim}$ such that $\Scontr(\opens)\defeq\bigcup_{i=1}^N\scontr_i\opens\subseteq\opens$, $\scontr_i\opens\cap\scontr_j\opens=\emptyset$ for $i\neq j\in\{1,\ldots,N\}$ and $O\cap K\neq\emptyset$. Such a set $\opens$ satisfies $\sset\subseteq\overline{O}$ and shall be fixed from now on. Motivated by Section 6.1 in \cite{Winter_thesis} we define
\begin{align*}
	\mathcal{E}_{\cset}&\defeq\{\map\scontr_{\om}\opens\mid\om\in\ABC^*\}\cup\mathcal{K}_{\cset},\quad\text{where}\\
  \mathcal{K}_{\cset}&\defeq\{C\in\Borel\mid\exists n\in\mathbb N\colon C\subseteq\mathbb R^{\edim}\setminus\bigcup_{\om\in\ABC^n}\map\scontr_{\om}\opens\}.
\end{align*} 
\begin{lemma}\label{lem:intersectiongenerator}
  $\mathcal{E}_{\cset}$ is an intersection stable generator for the Borel $\sigma$-algebra $\Borel$. 
\end{lemma}
\begin{proof}
  It can easily be seen that $\mathcal{E}_{\cset}$ is intersection stable and that $\mathcal{E}_{\cset}\subseteq\Borel$. Thus, what remains to show is that $\Borel\subseteq\sigma(\mathcal{E}_{\cset})$, where $\sigma(\mathcal{E}_{\cset})$ denotes the $\sigma$-algebra generated by $\mathcal{E}_{\cset}$. For this inclusion we are going to prove that every open set $U\subseteq\mathbb R^{\edim}$ is contained in the $\sigma$-algebra $\sigma(\mathcal{E}_{\cset})$.
In the proof of Lemma 6.1.1 in \cite{Winter_thesis} it is shown that every open set in $\mathbb R^{\edim}$ is a countable union of sets in 
\begin{align*}
	\mathcal{E}_{\sset}&\defeq\{\scontr_{\om}\opens\mid\om\in\ABC^*\}\cup\mathcal{K}_{\sset},\quad\text{where}\\
	\mathcal{K}_{\sset}&\defeq\{C\in\Borel\mid\exists n\in\mathbb N\colon C\subseteq\mathbb R^{\edim}\setminus\bigcup_{\om\in\ABC^n}\scontr_{\om}\opens\}.
\end{align*}
Thus, there exist sets $A_i\in\mathcal{E}_{\sset}$, $i\in\mathbb N$, such that $\map^{-1}U=\bigcup_{i=1}^{\infty}A_i$.
If $A_i\in\mathcal{K}_{\sset}$, then there exists an $n\in\mathbb N$ such that $A_i\cap\bigcup_{\om\in\ABC^n}\scontr_{\om}\opens=\emptyset$. This implies that $\map A_i\cap\bigcup_{\om\in\ABC^n}\map\scontr_{\om}\opens=\emptyset$ and hence we have that $\map A_i\in\mathcal{K}_{\cset}$.
If, on the other hand, $A_i\in\{\scontr_{\om}\opens\mid\om\in\ABC^*\}$, then $\map A_i\in\{\map\scontr_{\om}\opens\mid\om\in\ABC^*\}$. Therefore, $\map A_i\in{\mathcal{E}}_{\cset}$ for all $i\in\mathbb N$ and $U=\bigcup_{i=1}^{\infty}\map A_i$.
\end{proof}

For the proof of Theorem \ref{thm:locMinkcresult}(\ref{locaverage}) we also require the following lemma, which is a weaker version of Lemma 5.2.1 in \cite{Winter_thesis}.
\begin{lemma}[Winter]\label{lem:inner}
  There exist constants $\conste,\constz,\constd>0$ such that for all $\e,\s$ satisfying $0<\e\leq\s\leq\constd$ we have
\[
\leb(\sset_{\e}\cap (O^c)_{\s})\leq\conste\e^{\edim-\hdim}\s^{\constz}.
\]
\end{lemma}

\begin{proof}[Proof of Theorem \ref{thm:locMinkcresult}(\ref{locaverage})]
  By Theorem \ref{thm:Winter} we know that if $\Scontr$ is nonlattice, then $\Mink(\sset,\cdot)$ exists. Thus by Theorem \ref{thm:locMinkcresult}(\ref{locexistence}) also $\Mink(\cset,\cdot)$ exists and $\textup{d}\Mink(\cset,\cdot)=\lvert\map'\circ\map^{-1}\rvert^{\hdim}\textup{d}\left(\map_{\star}\Mink(\sset,\cdot)\right)$. 
	Hence, the assertion follows in the nonlattice case, since the existence of the local Minkowski content clearly implies the existence of the local average Minkowski content, $\tilde{\Mink}(\sset,\cdot)=\Mink(\sset,\cdot)$ and $\tilde{\Mink}(\cset,\cdot)=\Mink(\cset,\cdot)$. This leaves the case that $\Scontr$ is $\arith$-lattice for some $\arith>0$, which we now prove. 
	
	Under the assumption that the average Minkowski content of $\cset$ exists (which we show later), the family of finite Borel measures
  \[
  \mathcal{P}\defeq\left\{\mu_T(\cdot)\defeq\lvert\ln T\rvert^{-1}\int_T^1\e^{\hdim-\edim}\leb(\cset_{\e},\cdot)\frac{\textup{d}\e}{\e}\mid T\in(0,1)\right\}
  \]
  is tight and bounded. Let $(T_n)_{n\in\mathbb N}$ denote a sequence in $(0,1)$ converging to $0$. Then by Prohorov's Theorem, there exists a subsequence $(T_{n_k})_{k\in\mathbb N}$ and a finite Borel measure $\overline{\mu}$ depending on the sequence $(n_k)_{k\in\mathbb N}$ such that $(\mu_{T_{n_k}})_{k\in\mathbb N}$ converges weakly to $\overline{\mu}$.
	We will show that $\overline{\mu}$ coincides for every such sequence $(n_k)_{k\in\mathbb N}$ with the measure $\mu$ which is given by
\begin{equation}\label{eq:measpushfwd}
\frac{\textup{d}\mu}{\textup{d}\left(\map_{\star}\tilde{\Mink}(\sset,\cdot)\right)}=\lvert\map'\circ\map^{-1}\rvert^{\hdim}.
\end{equation}
For this we use the fact that two measures which coincide on an intersection stable generator of $\Borel$ coincide on the whole $\sigma$-algebra $\Borel$. Thus, by Lemma \ref{lem:intersectiongenerator} it remains to show that $\lim_{k\to\infty}\mu_{T_{n_k}}(A)=\mu(A)$ for every $A\in\mathcal{E}_{\cset}$ and arbitrary $(n_k)_{k\in\mathbb N}$. (This also implies that the average Minkowski content of $\cset$ exists and thus that $\mathcal{P}$ is tight and bounded.) 
However, this follows from the statement that 
\begin{equation}\label{eq:XobenXunten}
	\overline{X}(A)=\underline{X}(A)=\mu(A)
\end{equation}
for all $A\in\mathcal{E}_{\cset}$, where 
\begin{align*}
	\overline{X}(A)
  &\defeq \limsup_{T\to 0}\lvert \ln T\rvert^{-1}\int_{T}^1 \e^{\hdim-\edim}\leb(\cset_{\e}\cap A)\frac{\textup{d}\e}{\e}\quad\text{and}\\
	\underline{X}(A)
  &\defeq \liminf_{T\to 0}\lvert \ln T\rvert^{-1}\int_{T}^1 \e^{\hdim-\edim}\leb(\cset_{\e}\cap A)\frac{\textup{d}\e}{\e}.
\end{align*}
In order to demonstrate the equality in \cref{eq:XobenXunten}, let us start with the following observations.
If $\Scontr$ is $\arith$-lattice, then the function $t\mapsto (e^{-t})^{\hdim-\edim}\leb(\sset_{e^{-t}})$ converges along sequences of the form $(\arith n+x)_{n\in\mathbb N}$, where $x\in[0,\arith)$. This has been obtained in Equation (2.9) of \cite{Gatzouras} and results from renewal theory. Thus there exists a periodic function $f\colon\mathbb R^+\to\mathbb R^+$ with period $\arith$ such that for all $x\in[0,\arith)$
  \[
  \lim_{m\to\infty}(e^{-(x+m\arith)})^{\hdim-\edim}\leb(\sset_{e^{-(x+m\arith)}})=f(x).
  \]
Moreover, Equation (2.10) in \cite{Gatzouras}, which follows from Lebesgue's Dominated Convergence Theorem, states that 
    \[
    \lim_{m\to\infty}\int_{[0,\arith)}(e^{-(x+m\arith)})^{\hdim-\edim}\leb(\sset_{e^{-(x+m\arith)}})\textup{d}x
      =\int_{[0,\arith)}f(x)\textup{d}x.
    \]
  Thus, for an arbitrary $\de>0$ there exists an $M\in\mathbb N$ such that for all $m\geq M$ we have
\begin{equation}\label{ftheta}
  \left\lvert \int_{[0,\arith)}(e^{-(x+m\arith)})^{\hdim-\edim}\leb(\sset_{e^{-(x+m\arith)}})\textup{d}x-\int_{[0,\arith)}f(x)\textup{d}x\right\rvert<\de.
\end{equation}
For this $\de>0$ fix $M$ as above, take $\e_0(\de)$ as in Lemma \ref{lem:disj} and set 
\begin{align*}
	\underline{\sratio}&\defeq\min\{\sratio_{\om}\mid \om\in\codes(\e_0(\de),\de)\}\qquad\text{and}\\
	L&\defeq\max\{M-\ln(\low\underline{\sratio}),-\ln\e_0(\de), -\ln\constd\low\underline{\sratio}\},
\end{align*}
where $\constd$ is the constant from Lemma \ref{lem:inner}.
Denote by $\lfloor x\rfloor$ the integer part of $x\in\mathbb R$, that is, the largest integer which is less than or equal to $x$. Then we can reformulate the expressions $\overline{X}(A)$ and $\underline{X}(A)$ for $A\in\mathcal{E}_{\cset}$ as follows. 
\begin{align}
  \overline{X}(A)
  &=\limsup_{T\to\infty}T^{-1}\int_0^{T}(e^{-t})^{\hdim-\edim}\leb(\cset_{e^{-t}}\cap A)\textup{d}t\nonumber\\
	&=\limsup_{T\to\infty}T^{-1}\underbrace{\sum_{k=0}^{\lfloor\arith^{-1}(T-L)-1\rfloor}\int_{T-(k+1)\arith}^{T-k\arith}(e^{-t})^{\hdim-\edim}\leb(\cset_{e^{-t}}\cap A)\textup{d}t}_{\eqdef U(T,A)},\label{Xoben}
\end{align}
where the last equality follows from the fact that $t\mapsto (e^{-t})^{\hdim-\edim}\leb(\cset_{e^{-t}}\cap A)$ is continuous and thus locally integrable. Analogously, one obtains that 
\begin{equation}\label{Xunten}
  \underline{X}(A)
  =\liminf_{T\to\infty}T^{-1}U(T,A).
\end{equation}
In order to show that $\overline{X}(A)=\underline{X}(A)$ for all $A\in\mathcal E_{\cset}$, we distinguish between the cases $A\in\mathcal{E}_{\cset}\setminus\mathcal{K}_{\cset}$ and $A\in\mathcal{K}_{\cset}$.\\

\textsc{Case} 1: $A\in\mathcal{E}_{\cset}\setminus\mathcal{K}_{\cset}$.\\
In this case there exists $\nu\in\ABC^*$ such that $A=g\scontr_{\nu}\opens$. Assume that $\de$ is sufficiently small that $n(\om)\geq n(\nu)$ for all $\om\in\codes(\e_0(\de),\de)$ and define the set of words
\[
\codes_{\nu}(\e,\de)\defeq\{\om\in\codeepsde\mid[\om]\subseteq[\nu]\}
\]
for $\e\in[0,\e_0(\de)]$, where $[\om]$ denotes the $\om$-cylinder set.

In the following suppose that $T>L$.
As $e^{-T+(k+1)\arith}\leq e^{-L}\leq\e_0(\de)$ holds for every $k\in\{0,\ldots,\lfloor\arith^{-1}(T-L)-1\rfloor\}$, Lemma \ref{lem:disj} implies that $\bigcup_{\om\in\codes(e^{-T+(k+1)\arith},\de)}(\map\scontr_{\om}\sset)_{e^{-t}}$ is a disjoint union for every $k\in\{0,\ldots,\lfloor\arith^{-1}(T-L)-1\rfloor\}$ and $t\in (T-(k+1)\arith,T-k\arith]$. Therefore, for sufficiently large $T$ we obtain that
\begin{align}
  &U(T,\map\scontr_{\nu}\opens)\nonumber\\
  =\ & \sum_{k=0}^{\lfloor\arith^{-1}(T-L)-1\rfloor}\int_{T-(k+1)\arith}^{T-k\arith}(e^{-t})^{\hdim-\edim}\hspace{-3ex}\sum_{\om\in\codes(e^{-T+(k+1)\arith},\de)}\hspace{-1.5ex}\leb((g\scontr_{\om}\sset)_{e^{-t}}\cap\map\scontr_{\nu}\opens)\textup{d}t\nonumber\\
  \leq\ & \sum_{k=0}^{\lfloor\arith^{-1}(T-L)-1\rfloor}\hspace{-3ex}\sum_{\om\in\codes_{\nu}(e^{-T+(k+1)\arith},\de)}\hspace{-1.5ex}\gommin(e^{-T+(k+1)\arith})^{\edim}(1+\de)^{\edim}\sratio_{\om}^{\edim}\nonumber\\
  &\qquad\qquad\qquad	\cdot\int_{T-(k+1)\arith}^{T-k\arith}(e^{-t})^{\hdim-\edim}\leb(\sset_{e^{-t}/(\gommin(e^{-T+(k+1)\arith})\sratio_{\om})})\textup{d}t\nonumber\\
  =\ &\sum_{k=0}^{\lfloor\arith^{-1}(T-L)-1\rfloor}\hspace{-3ex}\sum_{\om\in\codes_{\nu}(e^{-T+(k+1)\arith},\de)}\hspace{-1.5ex} \gommin(e^{-T+(k+1)\arith})^{\hdim}(1+\de)^{\edim}\sratio_{\om}^{\hdim}\nonumber\\
  &\qquad\qquad\qquad	\cdot\int_{T-(k+1)\arith+\ln(\gommin(e^{-T+(k+1)\arith})\sratio_{\om})}^{T-k\arith+\ln(\gommin(e^{-T+(k+1)\arith})\sratio_{\om})}(e^{-t})^{\hdim-\edim}\leb(\sset_{e^{-t}})\textup{d}t.\label{eq:UTgpO}
\end{align}
Now observe that $T-\lfloor\arith^{-1}(T-L)\rfloor\arith+\ln(\gommin(e^{-T+(k+1)\arith})\sratio_{\om})\geq L+\ln(\low\underline{\sratio})\geq M$. Thus, we can apply \cref{ftheta} to obtain
\begin{align*} 
  &U(T,\map\scontr_{\nu}\opens)\\
	\leq\ & \sum_{k=0}^{\lfloor\arith^{-1}(T-L)-1\rfloor}\hspace{-4.5ex}\sum_{\om\in\codes_{\nu}(e^{-T+(k+1)\arith},\de)}\hspace{-2.5ex}\gommin(e^{-T+(k+1)\arith})^{\hdim}(1+\de)^{\edim}\sratio_{\om}^{\hdim}	\left(\int_{T-(k+1)\arith}^{T-k\arith}f(y)\textup{d}y+\de\right)\\
  \leq\ & \sum_{k=0}^{\lfloor\arith^{-1}(T-L)-1\rfloor}\sum_{\om\in\codes_{\nu}(0,\de)}\gommin(0)^{\hdim}(1+\de)^{\edim+\hdim}\sratio_{\om}^{\hdim}\left(\int_{T-(k+1)\arith}^{T-k\arith}f(y)\textup{d}y+\de\right) \\
  =\ &\sum_{\om\in\codes_{\nu}(0,\de)}\gommin(0)^{\hdim}(1+\de)^{\edim+\hdim}\sratio_{\om}^{\hdim}\left(\int_{T-\lfloor\arith^{-1}(T-L)\rfloor\arith}^{T}f(y)\textup{d}y+\de\right).
\end{align*}
We know, in light of Theorem \ref{thm:Gatzouras}, that the average Minkowski content of the self-similar set $\sset$ exists. In view of Equation (\ref{Xoben}), the upper estimate for $U(T,\map\scontr_{\nu}\opens)$ implies that
\begin{equation}\label{eq:Xover}
  \overline{X}(\map\scontr_{\nu}\opens)\leq(\tilde{\Mink}(\sset)+2\de)\sum_{\om\in\codes_{\nu}(0,\de)}\gommin(0)^{\hdim}(1+\de)^{\edim+\hdim}\sratio_{\om}^{\hdim}.
\end{equation}
Now, we focus on finding a lower bound. In analogy to \cref{eq:UTgpO} we have
\begin{align}
 U(T,\map\scontr_{\nu}\opens)
  \geq&\hspace{-1ex} \sum_{k=0}^{\lfloor\arith^{-1}(T-L)-1\rfloor}\hspace{-3ex}\sum_{\om\in\codes_{\nu}(e^{-T+(k+1)\arith},\de)}\hspace{-1.5ex}\gommin(e^{-T+(k+1)\arith})^{\edim}\label{Ulower}\\
	&\quad\cdot\int_{T-(k+1)\arith}^{T-k\arith}\hspace{-1.5ex}(e^{-t})^{\hdim-\edim}\leb((\scontr_{\om}\sset)_{e^{-t}/(\gommin(e^{-T+(k+1)\arith})(1+\de))}\cap\scontr_{\nu}O)\textup{d}t.\nonumber
	\end{align}
For $r>0$ we denote the \emph{inner $r$-parallel neighbourhood} of a set $A\subseteq\mathbb R^{\edim}$ by
\[
A_{-r}\defeq\{x\in A\mid\met(x,\partial A)>r\}
\]
and observe that $\leb(Y\cap U)\geq\leb(Y\cap U_{-r})\geq\leb(Y)-\leb(Y\cap (U^c)_r)$ for $Y,U\subseteq\mathbb R^{\edim}$, $U$ open and $r>0$.
Using Lemma \ref{lem:inner} with the constants $\conste,\constz,\constd$ fixed therein and that $T-(k+1)\arith\geq L\geq-\ln\constd\low\underline{\sratio}$ for all $k\in\{0,\ldots,\lfloor\arith^{-1}(T-L)-1\rfloor\}$ we obtain the following for all $\om\in\codes_{\nu}(e^{-T+(k+1)\arith},\de)$ and $t\in(T-(k+1)\arith,T-k\arith]$. To shorten the notation, we write $\mathfrak{G}\defeq\gommin(e^{-T+(k+1)\arith})(1+\de)$.
\begin{align*}
  &\leb((\scontr_{\om}\sset)_{e^{-t}/(\gommin(e^{-T+(k+1)\arith})(1+\de))}\cap\scontr_{\nu}O)\\
  \geq\ & \underbrace{\leb((\scontr_{\om}\sset)_{e^{-t}/(\gommin(e^{-T+(k+1)\arith})(1+\de))})}_{\eqdef A_1(t,\om,k)\eqdef A_1}
  -\leb((\scontr_{\om}\sset)_{e^{-t}/\mathfrak{G}}\cap(\scontr_{\nu}O^c)_{\ee^{-T+(k+1)\arith}/\mathfrak{G}})\\
	\geq\ & A_1 -\leb((\scontr_{\om}\sset)_{e^{-t}/\mathfrak{G}}\cap(\scontr_{\om}O^c)_{\ee^{-T+(k+1)\arith}/\mathfrak{G}})\\
	\geq\ & A_1 -\sratio_{\om}^{\edim}\leb(\sset_{e^{-t}/(\mathfrak{G}\sratio_{\om})}\cap O^c_{\ee^{-T+(k+1)\arith}/(\mathfrak{G}\sratio_{\om})})\\
	\geq\ & A_1 -\underbrace{\sratio_{\om}^{\edim}\conste\left(\frac{\ee^{-t}}{\gommin(e^{-T+(k+1)\arith})(1+\de)\sratio_{\om}}\right)^{\edim-\hdim}
\left(\frac{\ee^{-T+(k+1)\arith}}{\gommin(e^{-T+(k+1)\arith})(1+\de)\sratio_{\om}}\right)^{\constz}}_{\eqdef A_2(t,\om,k)}.
\end{align*}
For $i\in\{1,2\}$ define 
\begin{align*}
	B_i
	\defeq\sum_{k=0}^{\lfloor\arith^{-1}(T-L)-1\rfloor}\hspace{-4.5ex}\sum_{\om\in\codes_{\nu}(e^{-T+(k+1)\arith},\de)}\hspace{-2ex}\gommin(e^{-T+(k+1)\arith})^{\edim}\int_{T-(k+1)\arith}^{T-k\arith}(e^{-t})^{\hdim-\edim} A_i(t,\om,k)\textup{d}t
\end{align*}
and note that $U(T,\map\scontr_{\nu}\opens)\geq B_1-B_2$ holds by \cref{Ulower}. Further, recall that $T-\lfloor\arith^{-1}(T-L)\rfloor\arith+\ln(\gommin(\ee^{-T+(k+1)\arith})\sratio_{\om})\geq M$. Thus, for $i=1$ we have
\begin{align*}
  B_1
  &= \sum_{k=0}^{\lfloor\arith^{-1}(T-L)-1\rfloor}\hspace{-3.5ex}\sum_{\om\in\codes_{\nu}(e^{-T+(k+1)\arith},\de)}\hspace{-1.5ex}\gommin(e^{-T+(k+1)\arith})^{\edim}\sratio_{\om}^{\edim}\\
	&\qquad\qquad\cdot \int_{T-(k+1)\arith}^{T-k\arith}(e^{-t})^{\hdim-\edim}\leb(\sset_{e^{-t}/(\gommin(e^{-T+(k+1)\arith})\sratio_{\om}(1+\de))})\textup{d}t\\
  &=\sum_{k=0}^{\lfloor\arith^{-1}(T-L)-1\rfloor}\hspace{-3.5ex}\sum_{\om\in\codes_{\nu}(e^{-T+(k+1)\arith},\de)}\hspace{-1.5ex} \gommin(e^{-T+(k+1)\arith})^{\hdim}(1+\de)^{\hdim-\edim}\sratio_{\om}^{\hdim}\\
  &\qquad\qquad\cdot\int_{T-(k+1)\arith+\ln(\gommin(e^{-T+(k+1)\arith})\sratio_{\om}(1+\de))}^{T-k\arith+\ln(\gommin(e^{-T+(k+1)\arith})\sratio_{\om}(1+\de))}(e^{-t})^{\hdim-\edim}\leb(\sset_{e^{-t}})\textup{d}t \\
  &\stackrel{(\ref{ftheta})}{\geq}\sum_{k=0}^{\lfloor\arith^{-1}(T-L)-1\rfloor}\hspace{-3.5ex}\sum_{\om\in\codes_{\nu}(e^{-T+(k+1)\arith},\de)}\hspace{-1.5ex}\gommin(e^{-T+(k+1)\arith})^{\hdim}(1+\de)^{\hdim-\edim}\sratio_{\om}^{\hdim}\\
	&\qquad\qquad\cdot\left(\int_{T-(k+1)\arith}^{T-k\arith}f(y)\textup{d}y-\de\right)\\
  &\geq \sum_{k=0}^{\lfloor\arith^{-1}(T-L)-1\rfloor}\hspace{-3.5ex}\sum_{\om\in\codes_{\nu}(0,\de)}\hspace{-1.5ex}\gommin(0)^{\hdim}(1+\de)^{-\edim}\sratio_{\om}^{\hdim}\left(\int_{T-(k+1)\arith}^{T-k\arith}f(y)\textup{d}y-\de\right) \\
  &=\sum_{\om\in\codes_{\nu}(0,\de)}\gommin(0)^{\hdim}(1+\de)^{-\edim}\sratio_{\om}^{\hdim}\left(\int_{T-\lfloor\arith^{-1}(T-L)\rfloor\arith}^{T}f(y)\textup{d}y-\de\right).
\end{align*}
Setting $\upp\defeq\max_{x\in\sset_{1/2}}\lvert\map'(x)\rvert$, for $i=2$ we obtain that
\begin{align*}
	B_2
	&= \sum_{k=0}^{\lfloor\arith^{-1}(T-L)-1\rfloor}\hspace{-3.5ex}\sum_{\om\in\codes_{\nu}(e^{-T+(k+1)\arith},\de)}\hspace{-1.5ex}\gommin(e^{-T+(k+1)\arith})^{\hdim-\constz}\sratio_{\om}^{\hdim-\constz}\ee^{(-T+(k+1)\arith)\constz}\\
	&\qquad\qquad\cdot \underbrace{\conste(1+\de)^{\hdim-\edim-\constz}\arith}_{\eqdef\tilde{\tilde c}}\\
	&\leq \underbrace{\tilde{\tilde c}\cdot\upp^{\hdim}\low^{-\constz}}_{\eqdef\tilde c}\hspace{-1.5ex}\sum_{k=0}^{\lfloor\arith^{-1}(T-L)-1\rfloor}\hspace{-4.5ex}\sum_{\om\in\codes_{\nu}(e^{-T+(k+1)\arith},\de)}\hspace{-3ex}\sratio_{\om}^{\hdim}\cdot(\sratio_{\min}b_{\ee^{-T+(k+1)\arith},\de})^{-\constz}\cdot\ee^{(-T+(k+1)\arith)\constz}\\
	&\leq\tilde{c}\sratio_{\min}^{-\constz}\sum_{k=0}^{\lfloor\arith^{-1}(T-L)-1\rfloor}b_{\ee^{-L},\de}^{-\constz}\cdot\ee^{(-T+(k+1)\arith)\constz}\\
	&\leq\tilde{c}\sratio_{\min}^{-\constz}\cdot b_{\ee^{-L},\de}^{-\constz}\cdot\frac{\ee^{\constz(T-L)}-1}{1-\ee^{-\arith\constz}}\cdot\ee^{-T\constz}.
\end{align*}
Since by \cref{Ulower} we have that $U(T,\map\scontr_{\nu}\opens)\geq B_1-B_2$, it follows that 
\begin{align*}
	U(T,\map\scontr_{\nu}\opens)
	\geq&\sum_{\om\in\codes_{\nu}(0,\de)}\gommin(0)^{\hdim}(1+\de)^{-\edim}\sratio_{\om}^{\hdim}\left(\int_{T-\lfloor\arith^{-1}(T-L)\rfloor\arith}^{T}f(y)\textup{d}y-\de\right)\\
	&\qquad -\tilde{c}\sratio_{\min}^{-\constz}\cdot b_{\ee^{-L},\de}^{-\constz}\cdot\frac{\ee^{\constz(T-L)}-1}{1-\ee^{-\arith\constz}}\cdot\ee^{-T\constz}.
\end{align*}
By Theorem \ref{thm:Gatzouras} the average Minkowski content of the self-similar set $\sset$ exists and in view of Equation (\ref{Xunten}), the lower estimate for $U(T,\map\scontr_{\nu}\opens)$ implies that 
\begin{equation}\label{eq:Xunder}
  \underline{X}(\map\scontr_{\nu}\opens)\geq (\tilde{\Mink}(\sset)-2\de)\sum_{\om\in\codes_{\nu}(0,\de)}\gommin(0)^{\hdim}(1+\de)^{-\edim}\sratio_{\om}^{\hdim}.
\end{equation}
Since Equations (\ref{eq:Xover}) and (\ref{eq:Xunder}) hold for all $\de>0$, taking the limit as $\de$ tends to 0 implies
\begin{equation*}
  \overline{X}(\map\scontr_{\nu}\opens)
  \leq\tilde{\Mink}(\sset) \int_{\scontr_{\nu}\sset}\lvert\map'\rvert^{\hdim}\textup{d}\mu_{\hdim}
  \leq \underline{X}(\map\scontr_{\nu}\opens).
\end{equation*}
Hence, by Theorem \ref{thm:Winter} we have
\[
\overline{X}(\map\scontr_{\nu}\opens)
=\underline{X}(\map\scontr_{\nu}\opens)
=\int_{\map\scontr_{\nu}\sset}\lvert\map'\circ\map^{-1}\rvert^{\hdim}\textup{d}\left(\map_{\star}\tilde{\Mink}(\sset,\cdot)\right)
=\mu(\map\scontr_{\nu}\opens),
\]
where the last equality holds since the normalised $\hdim$-dimensional Hausdorff measure $\mathcal H_K$ on $\sset$ satisfies $\mathcal H_K(\scontr_{\nu}\opens)=\mathcal H_K(\scontr_{\nu}\sset)$ for all $\nu\in\ABC^*$.\\

\textsc{Case 2}: $A\in\mathcal{K}_{\cset}$.\\
In this case there exists an $n\in\mathbb N$ such that $A\subseteq\mathbb R^{\edim}\setminus\bigcup_{\om\in\ABC^n}\map\scontr_{\om}\opens$ and $A\in\Borel$. Setting $\upp\defeq\sup_{x\in\sset_{1/2}}\lvert\map'(x)\rvert$ as before, we have for such a set $A$ that
\begin{align*} 
  U(T,A)
  &= \sum_{k=0}^{\lfloor\arith^{-1}(T-L)-1\rfloor}\int_{T-(k+1)\arith}^{T-k\arith}(e^{-t})^{\hdim-\edim}\leb((g\sset)_{e^{-t}}\cap A)\textup{d}t\\
  &\leq \sum_{k=0}^{\lfloor\arith^{-1}(T-L)-1\rfloor}\int_{T-(k+1)\arith}^{T-k\arith}(e^{-t})^{\hdim-\edim}\upp^{\edim}\leb(\sset_{e^{-t}/\low}\cap\map^{-1} A)\textup{d}t\\
  &\leq \upp^{\edim}\int_{0}^{T}(e^{-t})^{\hdim-\edim}\leb(\sset_{e^{-t}/\low}\cap\map^{-1} A)\textup{d}t.
\end{align*}
Note that Equation (6.2.15) of \cite{Winter_thesis} states that 
\[
\lim_{T\to\infty}T^{-1}\int_0^T(\ee^{-t})^{\hdim-\edim}\leb(\sset_{\ee^{-t}}\cap C)\textup{d}t=0
\]
for every $C\in\mathcal{K}_{\sset}$. Hence, since $\map^{-1}A\in\mathcal{K}_{\sset}$,  we obtain via \cref{Xoben,Xunten} that
\[
\underline{X}(A)=\overline{X}(A)=0=\mu(A).
\]
\end{proof}

\end{document}